\documentclass{article}
\usepackage{amsmath}
\usepackage{graphicx}
\usepackage{amssymb}
\usepackage{amsthm}
\usepackage[english]{babel}
\usepackage{pgf}
\usepackage{pifont}  
\usepackage{paralist}
\usepackage{tikz}
\usetikzlibrary{intersections}

\newcommand{\lrbrace}[1]{\left \lbrace #1 \right \rbrace }
\newcommand{\lrrb}[1]{\left ( #1 \right ) }
\newcommand{\skp}[1]{ \left \langle #1 \right \rangle }
\newcommand{\abs}[1]{\left| #1 \right|}

\newcommand{\te}[1]{\text{#1}}

\newcommand{\set}[1]{\mathbb{#1}}

\newtheorem{theorem}{Theorem}[section]
  
  \newtheorem{proposition}[theorem]{Proposition}
  \newtheorem{lemma}[theorem]{Lemma}
  \newtheorem{corollary}[theorem]{Corollary}

\newcommand{\R}{\mathbb{R}}
\newcommand*\Hess{\mathop{\mathrm{Hess}}\nolimits}
\newcommand*\graph{\mathop{\mathrm{graph}}\nolimits}
\newcommand{\ga}{\gamma}
\newcommand{\pl}[2]{{\frac{\partial #1}{\partial #2}}}
\newcommand{\grad}{\nabla}
\newcommand{\embed}{\hookrightarrow}
\newcommand{\cone}{{\cal C}_N}

\begin{document}

\author{Sebastian Helmensdorfer, Peter Topping}
\title{The Geometry of Differential Harnack Estimates}
\date{\today\footnote{Paper invited by \emph{Actes du s\'eminaire Th\'eorie spectrale et g\'eom\'etrie}, Universit\'e de Grenoble.}}
\maketitle

\begin{abstract}
In this short note, we hope to give a rapid induction for non-experts into the world of Differential Harnack inequalities, which have been so influential in geometric analysis and probability theory over the past few decades. At the coarsest level, these are often mysterious-looking inequalities that hold for `positive' solutions of some parabolic PDE, and can be verified quickly by grinding out a computation and applying a maximum principle. In this note we emphasise the geometry behind the Harnack inequalities, which typically turn out to be  assertions of the convexity of some natural object. As an application, we explain how Hamilton's Differential Harnack inequality for mean curvature flow of a $n$-dimensional submanifold  of $\R^{n+1}$ can be viewed as following directly from
the well-known preservation of convexity under mean curvature flow, but this time of a $(n+1)$-dimensional submanifold of $\R^{n+2}$. We also briefly survey the earlier work that led us to these observations.
\end{abstract}

\section{Introduction}

Perhaps the simplest situation in which to explain Differential Harnack Estimates is that of the ordinary, scalar, linear heat equation on Euclidean space.
Let $u: \set{R}^n \times (0,T] \rightarrow (0, \infty)$ be a positive bounded solution of 
  \begin{equation}
    \label{eq:heatequ}
    \frac{ \partial u}{ \partial t }  = \triangle u.
  \end{equation}
R. Hamilton's matrix Harnack estimate \cite{hamiltonhe}, restricted to this special case, says that the solution $u$ 
satisfies the following differential inequality:
\begin{theorem}
\label{mhe}
For each $t\in (0,T]$, any positive, bounded solution $u$ to the heat equation satisfies
  \begin{equation}
    \label{eq:logconvheatequ}
    \Hess \lrrb{ \log u } + \frac{ I }{ 2t } \geq 0
  \end{equation}  
  where $I$ denotes the identity matrix. 
 \end{theorem}
Before we try to understand the geometry behind this inequality, let us try to understand why it is so useful.
Taking the trace of (\ref{eq:logconvheatequ})  yields the following special case of the seminal inequalities of Li-Yau \cite{liyau}, which predate the work of Hamilton:
\begin{equation}
  \nonumber \triangle \lrrb{ \log u } + \frac{ n }{ 2t } \geq 0,
\end{equation}
and by rewriting the heat equation \eqref{eq:heatequ}
as
\begin{equation}
  \nonumber \frac{ \partial }{ \partial t } \log u  = \triangle \lrrb{ \log u } + \abs{ \nabla \lrrb{ \log u } }^2,
\end{equation}
we find:

\begin{corollary}
  \label{thm:harnackheat}
  For each $t\in (0,T]$, any positive, bounded solution $u$ to the heat equation satisfies
  \begin{equation}
    \label{eq:harnackheat}
    \frac{ \partial }{ \partial t } \log u - \abs{\nabla \lrrb{ \log u } }^2 + \frac{ n }{ 2t } \geq 0.
  \end{equation}
\end{corollary}
In particular, this implies that 
$\frac{ \partial }{ \partial t } \log u \geq -  \frac{ n }{ 2t }$,
i.e. that $u$ cannot decrease too fast. Note that we have managed to get rid of all spatial derivatives of $u$.
To extract the greatest possible amount of information from 
Corollary \ref{thm:harnackheat}, pick two times $0<t_1<t_2\leq T$, and consider a smooth path $\ga:[t_1,t_2]\to\R^n$ from $x_1\in\R^n$ to $x_2\in\R^n$. We may then compute, also using Young's inequality,
\begin{equation}
\begin{aligned}
\frac{d}{dt}\log u (\ga(t),t) & = \pl{\log u}{t}+\langle\grad(\log u),\dot\ga\rangle\\
&\geq \left(\abs{\nabla \lrrb{ \log u } }^2 - \frac{ n }{ 2t }\right)
- \left(\abs{\nabla \lrrb{ \log u } }^2 + \frac{1}{4}|\dot\ga|^2
\right)\\
&=- \frac{ n }{ 2t } - \frac{1}{4}|\dot\ga|^2,
\end{aligned}
\end{equation}
and then integrate to find that
$$\log u(x_2,t_2)-\log u(x_1,t_1)\geq -\frac{n}{2}\log\left(\frac{t_2}{t_1}\right)-\frac{1}{4}\int_{t_1}^{t_2}|\dot\ga(t)|^2dt.$$
Because there is no longer any mention of $\ga$ on the left-hand side, we may now optimise this inequality by taking $\ga$ to be the minimising geodesic from $x_1$ to $x_2$, i.e. a straight line, 
to obtain
$$\log u(x_2,t_2)-\log u(x_1,t_1)\geq -\frac{n}{2}\log\left(\frac{t_2}{t_1}\right)-\frac{|x_2-x_1|^2}{4(t_2-t_1)},$$
and we have proved a classical Harnack estimate:
\begin{corollary}
  \label{cor:harnackheat}
  For $0 < t_1 < t_2 \leq T$ and $x_1,x_2 \in \set{R}^n$, any positive, bounded solution $u:\R^n\times (0,T]\to (0,\infty)$ to the heat equation satisfies:

  \begin{equation}
    \nonumber
     u \lrrb{ x_2, t_2 } \geq  u \lrrb{ x_1, t_1 } \lrrb{ \frac{ t_1 } { t_2 } }^{ \frac{ n }{ 2 } } \exp \lrrb{{- \frac{ \abs{x_2 - x_1}^2 }{ 4(t_2-t_1) } }}.
  \end{equation} 

\end{corollary}
This beautiful estimate tells us that positive solutions cannot decrease too quickly as time advances, even if we move a little in space. Moreover it is sharp, as can be seen by considering the fundamental solution of the heat equation
\begin{equation}
\rho(x,t)=\frac{1}{(4\pi t)^{n/2}}\exp\left(-\frac{|x|^2}{4t}\right).
\end{equation}
In particular, this sharpness is manifested in the fact that $\rho$ achieves equality in \eqref{eq:logconvheatequ}:
\begin{equation}
\label{rhoeq}
\Hess \lrrb{ \log \rho } + \frac{ I }{ 2t } \equiv 0.
\end{equation}
If we now subtract \eqref{rhoeq} from \eqref{eq:logconvheatequ}, we obtain the following simple geometric rephrasing of Theorem \ref{mhe}, which is entirely in the spirit of this note.
\begin{corollary}
\label{harn_cor}
For $t\in (0,T]$, the function $\log\left(\frac{u}{\rho}\right)$ is \emph{convex}.
\end{corollary}

In fact, with this formulation one can reduce the proof of the full matrix Harnack inequality in this case to the fact that the sum of  log-convex functions is again log-convex:

\begin{proof} (Corollary \ref{harn_cor}.)
By translating time by an arbitrarily small amount, we may assume that our solution is smooth on the whole of $\R^n\times [0,T]$, and $u(\cdot,0)$ is a positive function.
Fix $t>0$, and write
$$u(x,t)=\int_{\R^n} u(y,0)\rho(x-y,t)dy,$$
or alternatively 
$$F(x):=\frac{u(x,t)}{\rho(x,t)}=
\int_{\R^n} u(y,0)G(x,y)dy,$$
where $G(x,y):=\frac{\rho(x-y,t)}{\rho(x,t)}$.
We must therefore prove that $F$ is log-convex, i.e. that for all $x,z\in\R^n$ and $\alpha\in (0,1)$, we have 
$$F(\alpha x + (1-\alpha)z)\leq F(x)^\alpha F(z)^{1-\alpha}.$$
But for all $y\in\R^n$, the function $G(\cdot,y)$ is log-convex (even log-affine) and thus
\begin{equation}
\begin{aligned}
F(\alpha x + (1-\alpha)z)
&=\int_{\R^n} u(y,0)G(\alpha x + (1-\alpha)z,y)dy\\
&\leq\int_{\R^n} u(y,0)G(x,y)^\alpha G(z,y)^{1-\alpha}dy\\
&\leq
\left(\int_{\R^n}u(y,0)G(x,y)dy\right)^\alpha
\left(\int_{\R^n}u(y,0)G(z,y)dy\right)^{1-\alpha}\\
&= F(x)^\alpha F(z)^{1-\alpha}
\end{aligned}
\end{equation}
by H\"older.
\end{proof}

\section{The Differential Harnack estimate for mean curvature flow}

Now we consider solutions of another heat equation, namely of the mean curvature flow (see \cite{eckermcf}). Let $\lrrb{ M_t }_{t \in [0, T]} \subset \set{ R }^{n+1}$, $M_t = F_t \lrrb{ M^n }$ be a family of smoothly immersed hypersurfaces, satisfying the nonlinear PDE

\begin{equation}
  \label{eq:mcfequation}
  \frac{ \partial }{ \partial t } F_t = \vec{H} = - H \nu
\end{equation}
where $\vec{H}$ is the mean curvature vector, $\nu$ is a choice of unit normal and $H$ the corresponding mean curvature of $M_t$. We are especially interested in convex initial data $M_0$. Convexity is preserved along the flow and compact convex solutions shrink to a point at a finite time $T_{\max}$ (see \cite{huisken}).

We are also particularly interested in so-called \emph{self-expanders} of the mean curvature flow, which are hypersurfaces $M_1$ 
of Euclidean space that  solve the self-expander equation
\begin{equation}
  \label{eq:mcfselfexpander}
  H + \frac{ \skp{ x, \nu }}{ 2 } = 0
\end{equation}
on $M_1$, where again, $\nu$ is the unit normal such that $\vec{H}=-H\nu$. The family of hypersurfaces $M_t = \sqrt{ t } M_1$, $t > 0$, then provides a \emph{self-expanding} solution of (\ref{eq:mcfequation}) up to tangential diffeomorphisms (see \cite{eckermcf}).

Under the mean curvature flow, the mean curvature itself satisfies a parabolic equation, and one might hope to prove a Harnack inequality for $H$ under some sort of positivity condition.
R. Hamilton \cite{hamiltonmcf} showed that the correct positivity hypothesis is to ask for the solutions to be \emph{convex}, and proved an estimate that we state here only for compact submanifolds.
\begin{theorem}[Hamilton \cite{hamiltonmcf}]
  \label{thm:harnackmcf}
  Let $\lrrb{ M_t }_{t \in [0, T]}$ be a compact solution of (\ref{eq:mcfequation}) such that $M_0$ is convex. Then all $M_t$ are convex and satisfy for $t \in(0,T]$
  \begin{equation}
    \label{eq:harnackmcf}
    Z\lrrb{V, V} := \frac{ \partial H}{ \partial t }  + 2 \skp{ \nabla H, V } + h \lrrb{ V, V } + \frac{ H }{ 2 t } \geq 0
  \end{equation}
for any tangent vector $V$, where $\nabla$, $h$ and $\skp{ \cdot, \cdot }$ denote the induced connection on $M^n$, the second fundamental form of $M_t$ and the Euclidean inner product respectively. 
\end{theorem}
Again one can integrate this estimate along extremal curves to obtain a  Harnack inequality for the mean curvature in the classical sense (see also \cite{hamiltonmcf}).
\begin{corollary}
  \label{cor:harnackmcf}
  Under the assumptions of Theorem \ref{thm:harnackmcf} we have for $0 < t_1 < t_2 \leq T$, $x_1 \in M_{t_1}$ and $x_2 \in M_{t_2}$:

  \begin{equation}
    \nonumber
     H \lrrb{ x_2, t_2 } \geq  H \lrrb{ x_1, t_1 } \sqrt{ \frac{ t_1 } { t_2 } } e^{- \frac{ \Delta }{ 4 } }
  \end{equation}
where
\begin{equation}
  \nonumber 
   \Delta = \inf_\gamma \int_{t_1}^{t_2} \abs{ \frac{ d \gamma}{ dt }  }^2_{M_t} dt
\end{equation}
and the infimum is taken over all $C^1$-paths $\gamma:[t_1,t_2]\to\R^n$, which remain on the surface, i.e. $\ga(t)\in M_t$, 
and with $\gamma \lrrb{ t_1 } = x_1$ and $\gamma \lrrb{ t_2 } = x_2$. Here $\abs{ \frac{ d \gamma}{ dt }  }_{M_t}$ denotes the length of the component of the velocity vector of $\gamma$ that is tangent to $M_t$.
\end{corollary}

Harnack inequalities for geometric flows have numerous applications, e.g. 
Hamilton's Harnack estimate for mean curvature flow
can be applied to classify convex eternal solutions of the mean curvature flow, if the mean curvature assumes its space-time maximum, as translating solitons (see \cite{hamiltonmcf}). The original proof of Theorem \ref{thm:harnackmcf} directly uses a tensor maximum principle type argument.
As is the case for other equations, the first step to obtain Theorem \ref{thm:harnackmcf}
is to find the right quantity from \eqref{eq:harnackmcf}, and this 
was originally done by looking for expressions
that vanish on self-expanding solutions of \eqref{eq:mcfequation} and then  trying to combine these in an appropriate way. While this method has proved to be extremely effective, it does not give 
much geometric insight into Harnack expressions. Such an insight can be gained by considering appropriate \emph{space-time constructions} for geometric flows, as we now roughly describe.

The pioneering work in the direction of space-time constructions for the Ricci flow (see \cite{toppinglectures}) and for the mean curvature flow was done by B. Chow and S. Chu (see \cite{chowchu1, chowchu}). They managed to show that $Z \lrrb{V, V} - \frac{ H }{ 2 t }$ -- an expression that is constant along translating solitons -- approximately corresponds to the second fundamental form of 
an extreme stretching by a factor $N$ in the time direction of the so-called space-time track of the mean curvature flow.
One can take the limit as $N \rightarrow \infty$ of this second fundamental form, 
yielding exactly $Z \lrrb{V, V} - \frac{ H }{ 2 t }$. 

In the context of Ricci flow, B. Chow and D. Knopf developed this idea further by considering a form of rescaled Ricci flow in order to obtain a precise correspondence between the relevant Harnack quantity and the curvature of a degenerate space-time construction (see \cite{chowknopf} for further details). 
E. Cabezas-Rivas and P. Topping extended these ideas in \cite{estherpeter} by constructing a non-degenerate expanding space-time approximate Ricci soliton, the limit of whose curvatures gave the existing, and new, Harnack quantities. Thus Harnack inequalities correspond to the preservation of certain curvature conditions (positive curvature operator, positive complex sectional curvature etc.) under Ricci flow. In this note we will see the mean curvature flow analogue of these ideas, where the correspondence between preservation of `positive curvature' and Harnack inequalities turns out to be particularly clean and precise.

B. Kotschwar \cite{kotschwar} has recently 
considered variants of these ideas, 
recovering the Harnack quantities for a large class of curvature flows (the Harnack estimates are due to B. Andrews and K. Smoczyk, see \cite{andrewsharnack, smoczykharnack}), as the limit as $N \rightarrow \infty$ of the second fundamental form of 
certain stretched
variants of the space-time track.
Moreover he gave an alternative proof of Hamilton's Harnack estimate for the mean curvature flow by showing convexity of 
these space-time track variants.
This was achieved by proving a generalised tensor maximum principle,  applied on slices of the degenerate limit of 
the space-time track variants.

In the remainder of this paper, we give a rigorous, purely geometric 
proof of Hamilton's Harnack estimate \eqref{eq:harnackmcf}
which does not rely on any form of the maximum principle other than the classical preservation of convexity in solutions of mean curvature flow.
The mean curvature flow to which we apply this convexity-preservation is not the original flow $(M_t)$; instead we take the flow starting at a cone over the original initial surface $M_0$.
In particular, we never have to apply any maximum principle to any degenerate objects, and the whole proof is phrased purely in terms of convexity.

The remainder of the paper is organised as follows. In the next section we describe how to flow graphical space-time cones over hypersurfaces by their mean curvature. In Section \ref{sec:canharnack} we state the relationship between so-called
canonical self-expanders and the Harnack quantity for the mean curvature flow. Finally, in Section \ref{sec:geometricproof}, we put things together and show how preservation of convexity along the mean curvature flow directly yields the Harnack inequality, Theorem \ref{thm:harnackmcf}.

\section{Mean curvature flow of space-time cones}
\label{sec:mcfcones}

Let $\lrrb{M_t}_{t \in \left [0, T_{\max} \right )}$ be a compact convex solution of (\ref{eq:mcfequation}). By translating the flow within the ambient space, we may assume that $M_0$ encloses the origin in $\R^{n+1}$. For $N\geq 1$, we define the cone $\cone$ to be
the $(n+1)$-dimensional submanifold of $\R^{n+2}$ given by
\begin{equation}
  \nonumber
  \cone = \lrbrace{ t \lrrb{ x, N }: \; x \in M_0\embed\R^{n+1}, t \in [0,\infty) }. 
\end{equation}
We can view $\cone$ as the entire graph of a Lipschitz function $f_N:\R^{n+1}\to [0,\infty)$.
The cone $\cone$ becomes steeper as $N$ becomes larger and the Lipschitz constant of $f_N$ is bounded by $C(M_0)N$.

K. Ecker and G. Huisken \cite{eckerhuisken} developed a theory 
of mean curvature flow of entire graphs, that applies to the cone $\cone$. The essential properties of $\cone$ here are that it is the graph of a Lipschitz function $f_N$ on the whole of $\R^{n+1}$, and that it is `straight at infinity', which is guaranteed in particular by the fact that 
\begin{equation*}
\langle z, \nu^{\cone}\rangle = 0,
\end{equation*}
for all $z\in \cone$.
The Ecker-Huisken theory then implies the existence of a self-expander
$\tilde{ \Sigma }_N =  {\graph} (\tilde{v}_N)$, where 
$\tilde v_N:\R^{n+1}\to [0,\infty)$
has Lipschitz constant no greater than that of $f_N$ (which in turn is bounded by $C(M_0)N$, as we have observed) such that
$$\sqrt{t}\,\tilde\Sigma_N$$
defines (for $t>0$) a mean curvature flow starting at $\cone$ that is smooth away from the initial cone point.
(In particular, $\tilde \Sigma_N$ is asymptotic to $\cone$.)
See Figure \ref{fig:figure1}.

\begin{figure}[ht]
  \caption{The graphical self-expander $\tilde{\Sigma}_N$}
  \label{fig:figure1}
  
  \begin{tikzpicture}[scale=4]

  \draw  (-1,1) -- (0,0) -- (1,1);
  \draw  (-1,1) -- (-1.5,1.5);
  \draw  (1,1) -- node [very near end, below=11pt, fill=white] {$\cone$} (1.5, 1.5);

  \draw [dashed] (-1,1) .. controls (0.1,1.3) and (0,1.1) .. (1,1) .. controls (0.2,0.9) and (0,0.81) .. node [near start, above=19pt, fill=white] {$M_0$} (-1,1);

  \draw [thick] (-1.45,1.5) .. controls (0,0.2) .. node[very near end, above=9pt, fill=white] {$\tilde{\Sigma}_N$} (1.45,1.5);
  \draw [thick, dotted] (-1.45,1.5) .. controls (-0.1,1.6) and (0.1,1.55) .. (1.45,1.5) .. controls (0.1,1.35) and (0.05,1.4) .. (-1.45,1.5);

\end{tikzpicture} 


\end{figure}
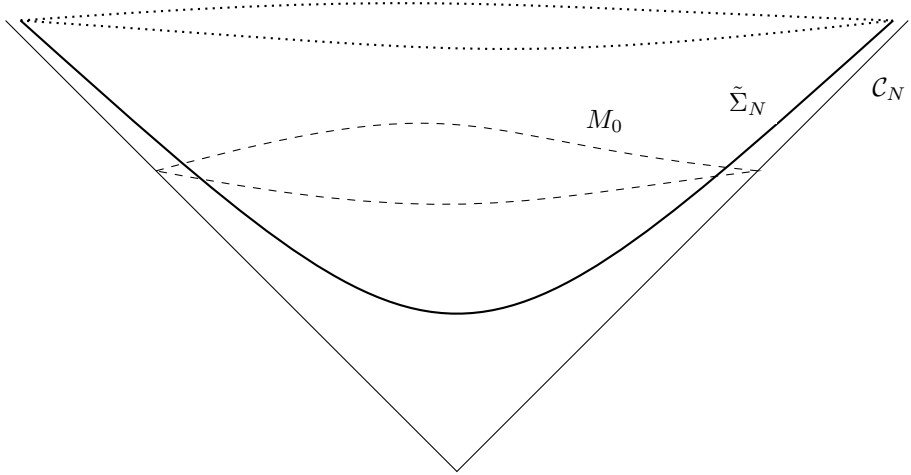

We can also say something useful about the infimum of $\tilde v_N$; more precisely we can argue that 
\begin{equation}
\label{min_bd}
\min \tilde v_N\leq C(M_0,n)N.
\end{equation}
To see this, let $d(M_0)>0$ be defined to be half the radius of the largest sphere centred at $(0,1)\in\R^{n+1}\times\R$ that does not intersect the region 
$$\lrbrace{ t \lrrb{ x, y }: \; x \in M_0\embed\R^{n+1}, t \in [0,\infty), y<1 }$$
below the cone $C_1$.
Then for all $N\geq 1$, we see that the $(n+1)$-sphere centred at $(0,N)\in\R^{n+1}\times\R$ of radius $d(M_0)$ will not intersect the cone $\cone$, or equivalently, that for all $h>0$, the 
sphere centred at $(0,h)\in\R^{n+1}\times\R$ of radius $\frac{h}{N}d(M_0)$ will not intersect the cone $\cone$. If we then evolve these spheres by mean curvature flow at the same time as we evolve $\cone$, the spheres will exist for a time $T_h:=\frac{h^2d^2}{2(n+1)N^2}$, 
(see, e.g. \cite{eckermcf}) and the comparison principle (see, e.g. 
\cite{eckermcf}) tells us that the evolution of $\cone$ cannot make it past the evolution of the spheres and go above the point $(0,h)$ for all $t\in [0,T_h]$.
In particular, we find that at time $t=1$, the evolution of $\cone$ cannot have gone above the point $(0,\sqrt{2(n+1)}\frac{N}{d})$, which is a statement a little stronger than \eqref{min_bd}.

In conclusion, we have established:
\begin{lemma}
\label{lem:selfexpexis}
For any $N \geq 1$ there exists a smooth graphical self-expander $\tilde{ \Sigma }_N =  {\graph} (\tilde{v}_N)$ with the same Lipschitz constant as $\cone$, in particular
$$\sup_{\R^{n+1}}\abs{ D \tilde{ v }_N  } \leq C(M_0) N,$$
and with controlled infimum
\begin{equation}
\label{inf_bd}
0\leq \min_{\set{R}^{n+1} } \tilde{v}_N  \leq C(M_0,n)N,
\end{equation}
such that $\sqrt{t}\,\tilde\Sigma_N$
defines (for $t>0$) a graphical mean curvature flow starting at $\cone$ that is smooth throughout the flow, except at the cone point at time $0$.
In particular, $\tilde \Sigma_N$ is asymptotic to $\cone$.
\end{lemma}

Since $\tilde{ \Sigma }_N =  \graph(\tilde{ v }_N)$ becomes steeper and steeper as $N \rightarrow \infty$, we squash it down by defining 
\begin{equation}
  \label{eq:squaseddowngraph}
  v_N := \frac{ 1 }{ N } \tilde{ v }_N.
\end{equation} 
Now we can take the limit of the functions $v_N$ as $N$ gets large, and level-set flow theory gives us a precise notion of what the limit is:

\begin{proposition}
 \label{prop:c0convergence}
There exists a sequence $N_n\to\infty$ such that $(v_{N_n})_{n \in \set{ N } }$ converges locally in $C^0$ to a continuous limit $v_\infty$. The graph of $v_\infty$ is the asymptotically conical space-time track defined to be
\begin{equation}
  \label{eq:expspacetimetrack}
  \lrbrace{ t^{- \frac{ 1 }{ 2 } } \lrrb{ x, 1 }: \; x \in M_t, t \in \left (0, T_{\max} \right ] }. 
\end{equation}
\end{proposition}

\begin{proof}

Lemma \ref{lem:selfexpexis} provides a suitable derivative bound for $\tilde{ v }_N$ in order to have a uniform bound on $ \abs{ D v _N }$, independently of $N$. Together with the infimum bound \eqref{inf_bd}
from Lemma \ref{lem:selfexpexis}, this implies that there is a sequence $N_n\to\infty$ such that $\lrrb{ v_{ N_n } }_{n \in \set{ N } }$  converges in $C^0_{loc} \lrrb{ \set{ R }^{n+1}}$ to a continuous limit $v_\infty$.

Lemma \ref{lem:selfexpexis} tells us that $\sqrt{t}\,\tilde\Sigma_N$ is a mean curvature flow that is the graph of a function that we will  call $\tilde V_N:\R^{n+1}\times (0,\infty)\to [0,\infty)$, with $\tilde V_N(\cdot,1)=\tilde v_N$, or more generally
$\tilde V_N(x,t)=\sqrt{t}\,\tilde v_N(x/\sqrt{t})$. 
We make the analogous extensions of $v_N$ and $v_\infty$ to $V_N$ and $V_\infty$ respectively.

The equation of graphical mean curvature flow (see for example 
\cite{eckermcf}) with respect to a time coordinate $s\in (0,\infty)$, is
\begin{equation}
  \nonumber
  \frac{ \partial }{ \partial s } \tilde{ V }_{N_n} = \sqrt{ 1 + \abs{ D \tilde{ V }_{N_n} }^2 } \te{ div } \lrrb{ \frac{ D \tilde{ V }_{N_n} }{ \sqrt{ 1 + \abs{ D \tilde{ V }_{N_n} }^2 }  }  },
\end{equation}
so we see that $V_{N_n}$ solves 
\begin{equation}
  \nonumber
  \frac{ \partial }{ \partial s }  V_{N_n} =  \sqrt{ \frac{ 1 }{ N_n^2 } + \abs{ D V_{N_n} }^2 }  \te{ div } \lrrb{ \frac{ D  V_{N_n} }{ \sqrt{ \frac{ 1 }{ N_n^2 } + \abs{ D V_{N_n} }^2 } }  }.
\end{equation}
We can now apply an approximation lemma (see \cite[proof of Theorem 4.2]{evansspruck}) in order to show that the limit 
$V_\infty$ corresponds to a weak solution of the level-set flow equation (see \cite{evansspruck, chengigagoto})
\begin{equation}
  \nonumber
   \frac{ \partial }{ \partial s } V_\infty = \lrrb{ I - \frac{ D V_\infty \otimes D V_\infty }{ \abs{ D V_\infty }^2 } } : D^2 V_\infty.  \\
\end{equation}

Thus $V_\infty$  is a weak solution of the level-set flow equation with the cone $C_1$ as an initial condition. Therefore for each 
height $\alpha>0$, 
the smooth $\alpha$-level sets of $V_\infty(\cdot,s)$ agree with the classical smooth evolution of the compact $\alpha$-level set of $C_1$ with respect to the time parameter $s$ (see \cite[Theorem 6.1]{evansspruck}). But a level set $\alpha M_0$ of $C_1$ at height 
$\alpha > 0$ gives rise to the parabolically rescaled evolution 
$\alpha M_{ \alpha^{-2}s}$, and by setting $s=1$, we see that the $\alpha$-level set of $v_\infty$ will be $\alpha M_{ \alpha^{-2}}$,
which is also the $\alpha$-level set of the asymptotically conical space-time track \eqref{eq:expspacetimetrack}. Hence the graph of $v_\infty$ must be the asymptotically conical space-time track as desired.
\end{proof}

\section{Canonical self-expanders and the Harnack quantity}
\label{sec:canharnack}

For a solution $\lrrb{M_t}_{t \in [0, T]}$ of (\ref{eq:mcfequation}) there is an associated space-time construction by B. Kotschwar (see \cite{kotschwar}), which we call a canonical self-expander by analogy with the Ricci flow case (see \cite{estherpeter}). The canonical self-expander $\Gamma_N$ can be defined for a
parameter $N > 0$ as
\begin{equation}
  \label{eq:canselfexp}
  \Gamma_N = \lrbrace{ t^{- \frac{ 1 }{ 2 } } \lrrb{ x, N }: \; x \in M_t, t \in (0,T] }.
\end{equation} 
Suppose now that the hypersurfaces $\lrrb{M_t}_{t \in [0, T]}$ have uniformly bounded curvature. Then $\Gamma_N$ is an approximate self-expander of (\ref{eq:mcfequation}), i.e.
\begin{equation}
  \label{eq:canonicalequation}
  H^{\Gamma_N} + \frac{ \skp{ z, \nu^{\Gamma_N} } }{ 2 } \approx 0
\end{equation}
for $z \in \Gamma_N$ (see \cite[4.1]{kotschwar}). By this we mean that $H^{\Gamma_N} + \frac{ \skp{ z, \nu^{\Gamma_N} } }{ 2 } = E_N$, where $N \abs{ E_N }$ is bounded locally uniformly, independently of $N$ (we continue to use this notation).

Furthermore $\Gamma_N$ is asymptotic to $\cone$ and we have (see 
Kotschwar \cite{kotschwar}):

\begin{theorem}
  \label{lem:canonicalsecondffharnack}
The second fundamental form of $\Gamma_N$, which we call 
$h^{\Gamma_N}$, satisfies
\begin{equation}
  \label{eq:canonicalsecondff}
  h^{ \Gamma_N } \lrrb{ V + \frac{ \partial } { \partial t }, V + \frac{ \partial } { \partial t } } = \frac{ Z \lrrb{ V, V } }{ \sigma_N \sqrt{t} } \approx \frac{ Z \lrrb{ V, V } }{ \sqrt{ t } }
\end{equation}
for any tangent vector $V \in TM^n$, where the constant $\sigma_N>0$
satisfies
$\sigma_N\to 1$ as $N\to\infty$.
\end{theorem}

\section{A geometric proof of Hamilton's Harnack estimate}
\label{sec:geometricproof}

We can now deduce Theorem \ref{thm:harnackmcf} directly from a chain of convexity statements
as follows:

\renewcommand{\labelitemi}{$\bigstar$} 

\begin{proof}
The starting assumption is that 
\begin{itemize}
\item
\emph{$M_0$ is convex,}
\end{itemize}
which immediately implies that 
\begin{itemize}
\item
\emph{The cone  $\cone$  over $M_0$ is convex.}
\end{itemize}
Since convexity is preserved under mean curvature flow
(see in particular \cite[Theorem 10.2]{barlesconvexity})
we find that
\begin{itemize}
\item
\emph{$\tilde{ \Sigma }_N = \graph(\tilde{v}_N)$  is convex,}
\end{itemize}
and since convexity is preserved under squashing in one direction, we deduce that
\begin{itemize}
\item
\emph{$\graph(v_N) = \graph(\frac{ 1 }{ N }\tilde{ v }_N)$
is convex.}
\end{itemize}
Now, $C^0_{loc}$-limits of convex functions are convex, and 
$v_\infty$ is the $C^0_{loc}$ limit of the convex functions $v_{N_n}$. Therefore
\begin{itemize}
\item
\emph{$v_\infty$ 
is convex,}
\end{itemize}
or equivalently, by Proposition \ref{prop:c0convergence},
\begin{itemize}
\item
\emph{The asymptotically conical space-time track \eqref{eq:expspacetimetrack} is convex,}
\end{itemize}
and by preservation of convexity under stretching in one direction, this implies that
\begin{itemize}
\item
\emph{The canonical self-expander $\Gamma_N$ is convex.}
\end{itemize}
By \eqref{eq:canonicalsecondff} we can then deduce the Harnack inequality
\begin{itemize}
\item
\emph{$Z \lrrb{ V, V } \geq 0$,}
\end{itemize}
as desired.
%
%
%
%
%
\end{proof}

\noindent
{\bf Acknowledgements}: 
This work was supported by The Leverhulme Trust.
The second author was also supported by EPSRC grant EP/K00865X/1.
Thanks to Klaus Ecker, Gerhard Huisken and Esther Cabezas-Rivas for conversations on this topic.

{\sc mathematics institute, university of Warwick, Coventry, CV4 7AL,
UK}\\

\end{document}